\documentclass[11pt, reqno]{amsart}
\usepackage{amsmath}
\usepackage{amssymb}
\usepackage{enumerate}
\usepackage[english]{babel}
\usepackage[all,cmtip]{xy}
\entrymodifiers={+!!<0pt,\fontdimen22\textfont2>}

\usepackage[a4paper,text={15.5cm,26cm},centering]{geometry}
\usepackage{hyperref}
\hypersetup{
  colorlinks,
  urlcolor=blue,
  pdftitle={},
  pdfauthor={Otgonbayar Uuye},
  pdfcreator={},
  pdfsubject={},
  pdfkeywords={}
  pdfhighlight={}
}
\newtheorem{thm}{Theorem}[section]
\newtheorem{lem}[thm]{Lemma}
\newtheorem{cor}[thm]{Corollary}
\newtheorem{prop}[thm]{Proposition}

\theoremstyle{definition}

\newtheorem{rem}[thm]{Remark}
\newtheorem{defn}[thm]{Definition}

\newtheorem{ex}[thm]{Example}

\newtheorem{notation}[thm]{Notation}

\newtheorem*{acknowledgments*}{Acknowledgments}

\numberwithin{equation}{section}

\mathchardef\ordinarycolon\mathcode`\: 
\def\vcentcolon{\mathrel{\mathop\ordinarycolon}} 
\providecommand*\coloneqq{\mathrel{\vcentcolon\mkern-1.2mu}=}

\mathchardef\mhyphen="2D

\def\Z{{\mathbb Z}} 
\def\Co{{\mathbb C}} 

\DeclareMathOperator*\colim{colim}

\newcommand\cat[1]{\mathbf{#1}}
\newcommand\Ho[1]{\cat{Ho}(#1)}
\newcommand\SW[1]{\cat{SHo}(#1)}
\def\CH{{\cat{CH}}}
\def\CM{{\cat{CM}}}
\def\Calg{{\cat{C^{*}}}}
\def\Csep{{\cat{sC^*}}}
\def\Aalg{{\cat{A^{*}}}}

\def\Top{{\cat{Top}}}

\newcommand\Mor[1]{{\mathrm{Mor}_{#1}}}

\def\inn{\in}
\def\ev{\mathrm{ev}}

\def\we{\overset{\sim}{\rightarrow}}

\def\fib{\twoheadrightarrow}

\def\afib{\overset{\sim}{\twoheadrightarrow}}

\def\C{\cat{C}}
\def\W{\cat{W}}
\def\F{\cat{F}}
\def\A{\cat{A}}
\def\D{\cat{D}}
\def\H{\mathcal H}
\def\bu{\cat{bu}}
\def\S{S}
\def\R{{\cat{R}}}
\def\P{{\cat{P}}}

\def\Ab{\cat{Ab}}

\def\KK{\cat{K\!K}}
\def\K{\cat{K}}

\def\M{\cat{M}}

\DeclareMathOperator\Hom{Hom}
\DeclareMathOperator\Ext{Ext}

\def\id{\mathrm{id}}
\def\pt{\ast}
\def\cast{$C^{*}$}

\newcommand\Cone[1]{{F}#1}
\newcommand\Cyl[1]{{N}#1}
\newcommand\MorHo[2]{[#2]_{#1}}

\def\tto{\rightrightarrows}

\def\Compact{\mathcal K}

\begin{document}

\title{Homotopical Algebra for \cast-algebras}
\author{Otgonbayar Uuye}
\date{\today}
\address{
School of Mathematics\\
Cardiff University\\
Senghennydd Road\\
CARDIFF, Wales, UK.\\
CF24 4AG}
\email{UuyeO@cardiff.ac.uk}

\begin{abstract}
Category of fibrant objects is a convenient framework to do homotopy theory, introduced and developed by Ken Brown. In this paper, we apply it to the category of \cast-algebras. In particular, we get a unified treatment of (ordinary) homotopy theory for \cast-algebras, $K\!K$-theory and $E$-theory,  as all of these can be expressed as the homotopy theory of a category of fibrant objects.
\end{abstract}

\keywords{category of fibrant objects, abstract homotopy theory, \cast-algebras, $K\!K$-theory, $K$-theory}
\subjclass[2010]{Primary (46L85); Secondary (55U35)}

\maketitle
\tableofcontents

\setcounter{section}{-1}
\section{Introduction}

Basic homotopy theory for \cast-algebras can be developed in an analogous way to the homotopy theory for topological spaces, using the Gelfand-Naimark duality between pointed compact Hausdorff spaces and abelian \cast-algebras. This is carried out, for example, by Rosenberg in \cite{MR658514} and Schochet in \cite{MR757510}. Thus, for instance, we have a version of the Puppe exact sequence, with essentially the same proof (cf.\ \cite[Proposition 2.6]{MR757510}).

There is one big difference: the homotopy theory for \cast-algebras
does not admit a Quillen model category structure, as first pointed
out by Andersen-Grodal (see Appendix A). This is unfortunate, since model categories provide a standard and powerful framework to study various aspects of homotopy theories. However, it turns out that not everything is lost: the category of \cast-algebras behave as if it was the subcategory of the fibrant objects in a model category, and this is enough for many purposes, because many proofs in model category theory start by reducing to the case of (co)fibrant objects. 

The notion of a ``category of fibrant objects'' is abstracted and developed by Brown in \cite{MR0341469}.  In this paper, we apply Brown's theory to the category of \cast-algebras. In Section \ref{sec AHT}, we review some basic facts about abstract homotopy theory in the setting of category of fibrant objects. 

In Section \ref{sec cast}, we first apply the abstract theory of Section \ref{sec AHT}  to the ordinary homotopy theory for \cast-algebras (this essentially recovers \cite{MR757510}). We also show that the Meyer-Nest's UCT category (cf.\ \cite{MR2193334}), Kasparov's $K\!K$-theory (cf.\ \cite{MR582160,MR918241}), and Connes-Higson's $E$-theory (cf.\ \cite{MR1068250,MR1065438}) can be described as the homotopy category of a category of fibrant objects. As a corollary, we get a unified treatment of the triangulated structures on these categories.

In addition to ordinary homotopy theory, we also have shape theories  for (separable) \cast-algebras (cf.\ \cite{MR866493, MR813640}). In \cite{MR1262931}, D{\u{a}}d{\u{a}}rlat constructed the {\em strong} shape category and showed that it is equivalent to the asymptotic homotopy category of separable \cast-algebras of Connes-Higson (cf.\ \cite{MR1065438}). 

Unfortunately and unlike the commutative case (cf.\ \cite{MR643532, MR601680}), we do not (yet) have a category of fibrant objects whose homotopy category describes the strong shape category. However, as we show in subsection \ref{subsec E}, the suspension-stable version considered by Thom (cf.\ \cite{andreasthomE}) does arise as the stable homotopy category of a category of fibrant objects. We also show that Thom's connective $K$-theory category fits well in this framework (cf.\ loc.cit.). 

Needless to say, Brown's theory of category of fibrant objects is not the only way to approach the homotopy theory for \cast-algebras. The main ``reason'' for the failure for the existence of a model structure on the category of \cast-algebras is that the category is too small, so an alternative approach would be to enlarge the category of \cast-algebras. Joachim-Johnson produced a model category structure for $K\!K$-theory by enlarging the category of \cast-algebras to a suitable category of topological algebras (cf.\ \cite{MR2222510}). {\O}stv{\ae}r developed a homotopy theory by enlarging the category of \cast-algebras to the category of \cast-spaces (cf.\ \cite{0812.0154}). Cuntz described an alternative construction of bivariant $K$-theories in \cite{MR1667652}.

We also note that Voigt computed the $K$-theory of free orthogonal quantum groups in \cite{Voigt20111873} using Meyer-Nest's triangulated category approach to the Baum-Connes conjecture (cf.\ \cite{MR2193334}). This seems to be the first concrete results in the theory of operator algebras, which can be proved only using abstract homotopy theoretic methods. 

Applications of the framework developed in this paper will appear elsewhere. 

\begin{acknowledgments*} This research was supported by the Danish National Research Foundation (DNRF) through the Centre for Symmetry and Deformation at the University of Copenhagen. I thank the referee for many useful suggestions and Ilan Barnea for pointing out a mistake in an earlier version.
\end{acknowledgments*}

\section{Abstract Homotopy Theory}\label{sec AHT}
For the convenience of the reader we recall some basic notions and results from abstract homotopy theory. See \cite{MR0223432, MR0341469, MR0224090, MR1464944, MR1711612} for details.

\subsection{Categories of Fibrant Objects}
The following is our main definition.
\begin{defn}[{Brown \cite{MR0341469}}]\label{defn cat Brown} 
Let $\C$ be category with terminal object $\pt$ and let $\F \subseteq \C$ and $\W \subseteq \C$ be distinguished subcategories. We say that $\C$ is a {\em category of fibrant objects} if the following conditions (F0) - (FW2) hold.

\begin{itemize}
\item[(F0)] The class $\F$ is closed under composition.
\item[(F1)] Isomorphisms of $\C$ are in $\F$.
\item[(F2)] The pullback in $\C$ of a morphism in $\F$ exists and is in $\F$.
\item[(F3)] For any object $B$ of $\C$, the morphism $B \to \pt$ is in $\F$.
\end{itemize}
Morphisms of $\F$ are called {\em fibrations} and denoted $\fib$.

\begin{itemize}
\item[(W1)] Isomorphisms of $\C$ are in $\W$.
\item[(W2)] If two of $f, g$ and $gf$ are in $\W$, then so is the third. 
\end{itemize}
Morphisms of $\W$ are called {\em weak equivalences} and denoted $\we$.

\begin{itemize}
\item[(FW1)] The pullback in $\C$ of a morphism in $\W \cap \F$ is in $\W \cap \F$.
\end{itemize}
Morphisms of $\W \cap \F$ are called {\em trivial fibrations} and denoted $\afib$.
\begin{itemize} 
\item[(FW2)] For any object $B$ of $\C$, the diagonal map $B \to B \times B$ admits a factorization 
	\begin{equation}\label{path-object decomposition}
	\xymatrix{B \ar[r]^{\sim}_-{s} & B^{I} \ar@{->>}[r]_-{d} & B \times B},
	\end{equation}
where $s \in \W$ is a weak equivalence, $d = (d_{0}, d_{1}) \in \F$ is a fibration.
\end{itemize}
The object $B^{I}$ or more precisely the quadruple $(B^{I}, s, d_0, d_1)$ is called a {\em path-object} of $B$.
\end{defn}

If there is no risk for confusion, we simply say that $\C$ is a category of fibrant objects. If the terminal object is also an {\em initial} object, we say that $\C$ is a {\em pointed} category of fibrant objects. 

\begin{rem} 
\begin{enumerate} 
\item The condition (F0) is superfluous since $\F$ is assumed to be a subcategory. But it is convenient to have a notation for this property.
\item The conditions (F1) and (W1) imply that $\F$ and $\W$ contain all objects of $\C$.
\item The conditions (F2) and (F3) imply that $\C$ is has finite products.
\end{enumerate}
\end{rem}

\begin{rem} Dually there is a notion of a {\em category of cofibrant objects}.
\end{rem}

The following is the motivating example.
\begin{ex}\label{ex model} For any model category $\M$,  the full subcategory $\M_{f}$ consisting of the fibrant objects in $\M$ is naturally a category of fibrant objects, by restricting the weak equivalences and the fibrations to $\M_{f}$. 

In particular, if $\Top$ denote the category of compactly generated weakly Hausdorff topological spaces and continuous maps, then
\begin{enumerate}
\item $\Top$, homotopy equivalences, Hurewicz fibrations;
\item $\Top$, weak homotopy equivalences, Serre fibrations;
\end{enumerate}
are examples of categories of fibrant objects. In this paper, we only consider the latter one.

A more algebraic example is the following:	let $R$ be a ring and let $\cat{Ch}(R)$ denote the category of chain complexes of left $R$-modules and chain maps. Then
\begin{enumerate}
\item[(3)] $\cat{Ch}(R)$, quasi-isomorphisms, degreewise epimorphisms
\end{enumerate}
is a category of fibrant objects. In these three examples, all objects are fibrant i.e.
 $\M_{f} = \M$.
\end{ex}

\begin{defn} A functor between categories of fibrant objects is said to be {\em exact} if it preserves all the relevant structure: it sends the terminal object to the terminal object, fibrations to fibrations, weak equivalences to weak equivalences and pullbacks (of fibrations) to pullbacks.
\end{defn}

\begin{ex}\label{ex ref subcat} Let $\C$ be a category of fibrant objects and let $\A \subseteq \C$ be a full {\em reflective} subcategory i.e.\ the inclusion $i\colon \A \to \C$ is a right-adjoint. Suppose that for any $B \inn \A$, a path-object $B^{I}$ can be chosen to lie in $\A$. Then $\A$ is a category of fibrant objects by restricting weak equivalences and fibrations, since limits in $\A$ can be computed in $\C$; and the inclusion $i\colon \A \to \C$ is exact.
\end{ex}

Occasionally, we find it convenient to isolate the notions of weak equivalences and fibrations.
\begin{defn}
Let $\C$ be a category. A {\em subcategory of weak equivalences} is a  subcategory $\W \subseteq \C$ satisfying (W1) and (W2).
If $\C$ has a terminal object, a {\em subcategory of fibrations} is a subcategory $\F \subseteq \C$ satisfying (F0) - (F3).
\end{defn}

\subsection{Fibre and Homotopy Fibre}
\begin{lem}[Factorization Lemma]\label{lem factorization} Let $f\colon A \to B$ be a morphism in a category of fibrant objects. Consider the diagram 
	\begin{equation}
	\xymatrix{
	\Cyl{f} \ar[r]^p \ar@{=}[d] & B\\
	\Cyl{f} \ar[r]^{d_{0}^{*}(f)} \ar[d] & B^{I} \ar[d]^{d_{0}} \ar[u]^{d_{1}} \\
	A \ar[r]^{f} \ar@/^/[u]^{i} & B \ar@/^/[u]^{s} 
	},
	\end{equation}
where $(B^{I}, s, d_0, d_1)$ is a path-object for $B$ and $\Cyl{f}$ is the pullback $A \times_{B}B^{I}$ and $p$ is the composition $d_{1} \circ d_{0}^{*}(f)$ and $i$ is the map determined by the section $s$.

Then $p$ is a fibration and $i$ is a right inverse to a trivial fibration (in particular, a weak equivalence) and $f = p\circ i$.
\end{lem}
\begin{proof}
\cite[Factorization Lemma]{MR0341469}.
\end{proof}
\begin{defn}
We call $\Cyl{f}$ a {\em mapping path-object} of $f$.
\end{defn}

\begin{cor} Let $\C$ be a category of fibrant objects and let $\D$ be a category with weak equivalences. Let $F \colon \C \to \D$ be a functor that sends trivial fibrations to weak equivalences. Then $F$ send weak equivalences to weak equivalences. \qed
\end{cor}

Now we consider pointed categories.

\begin{defn} Let $p$ be a fibration in a pointed category of fibrant objects. The {\em fibre} $F$ of $f$ is the pullback 
	\begin{equation}
	\xymatrix{
	F \ar[r]^{i} \ar[d] & E \ar[d]^{p}\\
	\pt \ar[r] & B
	}.
	\end{equation}
We express this situation by the diagram
	\begin{equation}
	\xymatrix{F \ar@{>->}[r]^{i}  & E \ar@{->>}[r]^{p} & B}.
	\end{equation}
	
\end{defn}

\begin{defn} Let $f\colon A \to B$ be a morphism in a pointed category of fibrant objects.	The {\em homotopy fibre} $\Cone{f}$ of $f$ is the fibre of $\Cyl{f} \overset{p}{\fib} B$, where $p$ is as in the Factorization Lemma (Lemma~\ref{lem factorization}). 
\end{defn}

\begin{lem}\label{lem fibre and homotopy fibre} Let $p$ be a fibration in a pointed category of fibrant objects with fibre $F$. Then the natural map 
	\begin{equation}
	F \longrightarrow \Cone{p}
\end{equation}
is a weak equivalence.
\end{lem}
\begin{proof} 
Apply \cite[Lemma 4.3]{MR0341469} to
	\begin{equation}
	\xymatrix{
	F \ar@{>->}[r] \ar[d] & E \ar[d]^{\wr} \ar@{->>}[r]^{p} & B \ar@{=}[d]\\
	Fp \ar@{>->}[r] & Np  \ar@{->>}[r] & B
	}.
\end{equation}
\end{proof}

\subsection{Homotopy Category}

\begin{notation}If $\C$ is a category, we write $\mathrm{Ob}\C$ for the objects of $\C$ and write $\Mor\C(A, B)$ for the space of morphisms from $A$ to $B$, for $A$, $B \inn \C$.
\end{notation}

\begin{defn} The {\em homotopy category} of a category $\C$ of fibrant objects with weak equivalences $\W$ is the localization 
	\begin{equation}
	\Ho{\C} \coloneqq \C[\W^{-1}].
	\end{equation}
\end{defn}
In other words, there is given a functor $\gamma\colon \C \to \Ho\C$, called the {\em localization functor}, with the property that for any functor $k\colon \C \to \D$ such that $k(t)$ is invertible in $\D$ for all $t \in \W$, there exist a unique functor $\Ho\C \to \D$ making the diagram
	\begin{equation}
	\xymatrix{
	& \Ho\C \ar@{-->}[rd] &\\
	\C \ar[ur]^{\gamma} \ar[rr]^k & & \D}
	\end{equation}
commute. 	

Often we write $\MorHo{\C}{A, B}$ for $\Mor{\Ho\C}(A, B)$. Note that there is no guarantee that $\MorHo{\C}{A, B}$ is a small set (see Corollary \ref{cor small set}).
\begin{defn}\label{Brown homotopy} Let $\C$ be a category of fibrant objects. Two morphisms 
	\begin{equation}f_{0}, f_{1}\colon A \tto B\end{equation}
are said to be {\em right-homotopic} if for some path-object $(B^{I}, s, d_{0}, d_{1})$ of $B$, there is a morphism $h\colon A \to B^{I}$ such that $f_{0} = d_{0}h$ and $f_{1} = d_{1}h$. 

The two are said to be {\em homotopic} if there is a weak equivalence $t\colon A' \to A$ such that $f_{0}t, f_{1}t\colon A' \tto B$ are right-homotopic.  
\end{defn}

Right-homotopy and homotopy are equivalence relations, and moreover, homotopy is compatible with the composition in $\C$ (cf.\ \cite[Section 2]{MR0341469}).

\begin{defn}
Let $\C$ be a category of fibrant objects. We denote the category of {\em homotopy classes} in $\C$ by $\pi\C$ and let $\pi\colon \C \to \pi\C$ denote the quotient functor. 
\end{defn}

The following is the fundamental result of Brown. 
\begin{thm}[{Brown \cite[Theorem 2.1]{MR0341469}}] Let $\C$ be a category of fibrant objects. Then $\pi\W \subseteq \pi\C$ admits a {\em calculus of right fractions}. 

It follows that, for $A$, $B \inn \C$
	\begin{equation}
	\MorHo{\C}{A, B} \cong \colim_{A' \we A} \Mor{\pi\C}(A', B)
	\end{equation}
and hence if $\gamma\colon \C \to \Ho\C$ is the localization functor, then
\begin{enumerate}
	\item any morphism in $\MorHo{\C}{A, B}$ can be written as a right-fraction 
	\begin{equation}
	\xymatrix{A & \ar[l]_{\,\,\gamma(t)^{-1}} A' \ar[r]^{\gamma(f)} & B}
	\end{equation}
where $t \in \W$ is a weak equivalence, and
	\item if $f_0, f_1$ are morphisms in $\Mor{\C}(A, B)$, then $\gamma(f_0) = \gamma(f_1)$ if and only if $f_0$ and $f_1$ are homotopic i.e.\ $\pi(f_{0}) = \pi(f_{1})$.
\end{enumerate}
\qed     
\end{thm}

\begin{cor}\label{cor small set} Let $\C$ be a category of fibrant objects and let $A$ be an object in $\C$. Suppose that the category $\W_{A}$ of weak equivalences over $A$ is ``coinitially small'' i.e\ there exists a {\em set} $S_{A}$ of objects in $\C$ such that for any $A' \we A$, there is a $A'' \we A'$ such that $A'' \in S_{A}$,
then $\MorHo{\C}{A, B}$ is a small set for every $B \inn \C$.
\qed
\end{cor}
\begin{proof} See \cite[Proposition 2.4]{MR0210125}.
\end{proof}

Now we consider pointed categories.
\begin{defn} Let $B$ be an object of a pointed category of fibrant objects. A {\em loop-object} of $B$ is the fibre $\Omega B$ of $(d_0, d_1)\colon B^{I} \to B \times B$, where $(B^{I}, s, d_0, d_1)$ is a path-object of $B$.
\end{defn}

\begin{lem}\label{lem loop object}  Let $\C$ be a pointed category of fibrant objects. Then $\Omega$ defines a functor 
	\begin{equation}
	\Omega\colon \Ho{\C} \to \Ho{\C},
	\end{equation}
called the {\em loop-object} functor.
\begin{enumerate}
\item For any $B \inn \C$, the object $\Omega B$ is naturally a group object in $\Ho\C$ and $\Omega^2 B$ is naturally an abelian group object in $\Ho\C$.
\item For any fibration $p\colon E \fib B$ with fibre $F$, there is a natural right-action $F \times \Omega B \to F$ in $\Ho\C$. In particular, we have a natural map $\Omega B \to F$ in $\Ho\C$.
\end{enumerate} 
\end{lem}
\begin{proof}
See \cite[Section 4]{MR0341469}.
\end{proof}

\begin{thm}\label{thm Puppe} Let $\C$ be a pointed category of fibrant objects and let $p\colon E \fib B$ be a fibration with fibre $F$. Then for any $D \in \C$, there is an exact sequence 
	\begin{equation*}
	\dots \to \MorHo{\C}{D, \Omega^{2}B} \to \MorHo{\C}{D, \Omega F} \to \MorHo{\C}{D, \Omega E} \to \MorHo{\C}{D, \Omega B} \to \MorHo{\C}{D, F} \to \MorHo{\C}{D, E} \to \MorHo{\C}{D, B}.
	\end{equation*}
\end{thm}
\begin{proof} See \cite[Section 4]{MR0341469} and \cite[Section I.3]{MR0223432}.
\end{proof}

Note that while $\Ho\C$ depends only on the weak equivalences, the loop-object functor $\Omega$ depends also on the fibrations.

\begin{defn}
Let $\C$ be a pointed category of fibrant objects. We define the {\em stable homotopy category} of $\C$ as the category 
	\begin{equation}
	\SW\C \coloneqq \Ho\C[\Omega^{-1}],
	\end{equation}
obtained from $\Ho\C$ by inverting the endofunctor $\Omega$.
\end{defn}

Objects of $\SW\C$ are $(A, n)$ with $A \inn \Ho\C$ and $n \in \Z$ and the morphisms are given by
	\begin{equation}
	\Mor{\SW\C}((A, n), (B, m)) \coloneqq \colim_{k \to \infty}\MorHo{\C}{\Omega^{n+k}A, \Omega^{m+k}B}.
	\end{equation} 

For $n \in \Z$, we have natural functors, also denoted $\Omega^n$, 
	\begin{equation}
	\Omega^n\colon \Ho\C \to \SW\C,\quad A \mapsto (A, n),
	\end{equation}
which sends morphisms in $\Mor{\Ho{\C}}{A, B}$ to the corresponding element in $\Mor{\SW\C}((A, n), (B, n))$.
\begin{thm}\label{thm fibrant object to triangulated} Let $\C$ be a pointed category of fibrant objects. Then the stable homotopy category $\SW\C$ is a {\em triangulated category} with the {\em shift} 
	\begin{equation}
	\Sigma = \Omega^{-1}\colon\SW\C \to \SW\C
	\end{equation}
given by $(A, n) \mapsto (A, n-1)$ and the {\em distinguished triangles} given by triangles isomorphic to triangles of the form
	\begin{equation}
	\xymatrix{(\Omega B, n) \ar[r] & (F, n) \ar[r] & (E, n) \ar[r] & (B, n)},
	\end{equation}
where $n \in \Z$ and $E \to B$ is a fibration, $F \to E$ is the fibre inclusion and $\Omega B \to F$ is the morphism obtained from Lemma~\ref{lem loop object}. 
\end{thm}
\begin{proof} See \cite{MR0224090} or \cite{MR1650134, MR1867203}.
\end{proof}

\begin{rem}\label{rem distinguished triangle} We note that for any $f \in \MorHo{\C}{A, B}$ and $n \in \Z$, we have a natural distinguished triangle
	\begin{equation}
	\xymatrix{(\Omega B, n) \ar[r] & (Ff, n) \ar[r] & (A, n) \ar[r]^{\Omega^{n}f} & (B, n).
	}
	\end{equation} 
\end{rem}
\begin{defn} We say that a pointed category of fibrant objects $\C$ is {\em stable}, if the loop functor $\Omega\colon \Ho\C \to \Ho\C$ is invertible. 
\end{defn}

\begin{rem}\label{prop stable => triangulated} 
If $\C$ is a {\em stable} pointed category of fibrant objects, then 
	\begin{equation}
	\Omega^0\colon \Ho\C \to \SW\C
	\end{equation}
is an equivalence of categories. In particular, $\Ho\C$ is naturally a triangulated category with shift $\Sigma = \Omega^{-1}\colon \Ho\C \to \Ho\C$. 
\end{rem}

\begin{ex} Let $\M$ be a pointed Quillen model category and let $\M_{f}$ be the full subcategory of fibrant objects in $\M$, considered a category of fibrant objects as in Example~\ref{ex model}. Then the inclusion $\M_{f} \to \M$ induces an equivalence $\Ho{\M_{f}} \cong \Ho\M$, with compatible loop-objects and fibration sequences. Compare \cite{MR0341469} and \cite{MR0223432}.
\end{ex}


\subsection{Homology Theories and Localizations}

\begin{defn}\label{defn fht}
A {\em homology theory} on a pointed category of fibrant objects $\C$ is a homology theory on $\SW\C$ i.e.\ an exact functor $\H\colon \SW\C \to \Ab$.
\end{defn}

\begin{defn} Let $\C$ be a pointed category of fibrant objects and let $\H$ be a homology theory on $\C$.

A morphism $t \colon A \to B$ is said to be an {\em $\H$-equivalence} if the induced maps 
	\begin{equation}
	(\Omega^{n}t)_{*} \colon \H(A, n) \to \H(B, n)
	\end{equation} are isomorphisms for all $n \in \Z$. 
	
An object $F \inn \C$ is said to be {\em $\H$-acyclic} if $\H(F, n) = 0$ for all $n \in \Z$.
\end{defn}

Note that since homology theories are homotopy invariant by definition, weak equivalences in $\C$ are $\H$-equivalences.

\begin{lem}\label{lem H-eq H-acyclic} Let $\C$ be a pointed category of fibrant objects and let $\H$ be a homology theory on $\C$. Then a morphism $t$ in $\C$ is an $\H$-equivalence if and only if its homotopy fibre $\Cone{t}$ is $\H$-acyclic.  
\end{lem}
\begin{proof} Clear from the long-exact sequence associated to the distinguished triangle of Remark~\ref{rem distinguished triangle}.
\end{proof}

\begin{cor}\label{cor fibration H-eq H-acyclic} Let $\C$ be a pointed category of fibrant objects and let $\H$ be a homology theory on $\C$. Then a fibration $p \in \C$ with fibre $F$ is an $\H$-equivalence if and only if $F$ is $\H$-acyclic. 
\end{cor}
\begin{proof} By Lemma~\ref{lem fibre and homotopy fibre}, the natural map $F \to \Cone{p}$ is a weak equivalence, hence an $\H$-equivalence. The proof is complete by Lemma~\ref{lem H-eq H-acyclic}.
\end{proof}

\begin{thm}\label{thm fibre homology fibrant object} Let $\C$ be a pointed category of fibrant objects and let $\H$ be a homology theory on $\C$. Then $\H$-equivalences and fibrations define a pointed category of fibrant objects on $\C$, denoted $R_{\H}\C$, with the same path and loop objects as in $\C$.
\end{thm}
\begin{proof} It is clear that $\H$-equivalences form a subcategory of weak equivalences. Hence we need to show the compatibility conditions (FW1) and (FW2) are satisfied.
  
(FW1) Let $p\colon E \fib B$ be a fibration which is also an $\H$-equivalence. We need to show that for any $f\colon A \to B$, the pullback $f^*(p)$ is again an $\H$-equivalence. But this is immediate from Corollary~\ref{cor fibration H-eq H-acyclic} applied to the diagram:
	\begin{equation}\label{eq pullback of fibration seq}
	\xymatrix{
	 F  \ar@{=}[d]  \ar@{>->}[r] & E \times_{B} A  \ar@{->>}[r]^-{f^{*}(p)} \ar[d] & A \ar[d]^{f}\\
	 F    \ar@{>->}[r] & E  \ar@{->>}[r]^{p} & B 
	 },
	\end{equation}
where $F$ is the fibre of $p$.

(FW2) Since weak equivalences are $\H$-equivalences, path-objects in $\C$ also give path-objects in the new category of fibrant objects $R_{\H}\C$. 
\end{proof}

\begin{defn}
Let $\C$ be a pointed category of fibrant objects and let $\S \subseteq \C$ be a class of morphisms. We say that a morphism $t \in \Mor\C(A, B)$ is a {\em $\S^{-1}$-weak equivalence} if for any homology theory $\H \colon \SW\C \to \Ab$ such that every $s \in \S$ is an $\H$-equivalence, $t$ is an $\H$-equivalence.
\end{defn}

\begin{thm}\label{thm localization} Let $\C$ be a pointed category of fibrant objects and let $\S \subseteq \C$ be a class of morphisms. Then $\S^{-1}$-weak equivalences and fibrations define a pointed category of fibrant objects, denoted $\R_{\S}\C$.  
The stable homotopy category $\SW{\R_{\S}\C}$ is naturally equivalent to the Verdier localization $\SW\C[(\Omega^{0}\S)^{-1}]$ as a triangulated category.
\end{thm}

\begin{proof} Considering all homology theories $\H\colon \SW\C \to \Ab$ in which every $t \in \S$ is an $\H$-equivalence in Theorem~\ref{thm fibre homology fibrant object}, we see that $\R_{\S}\C$ is indeed a category of fibrant objects.

Now consider the natural triangulated functor $Q\colon \SW\C \to \SW{\R_{\S}\C}$ induced by $\C \to \R_{\S}\C$. Since any $s \in \S$ is a $\S^{-1}$-weak equivalence, we see that $Q(\Omega^{0}s)$ is invertible in $\SW{\R_{\S}\C}$.

We show that $Q$ is the universal triangulated functor that invert $\Omega^0\S \subseteq \SW\C$. Indeed, let $R\colon \SW\C \to (\P, \Omega^{-1})$ be a triangulated functor such that morphisms in $R(\Omega^0\S) \subseteq \Mor{\P}(R(A, 0), R(B, 0))$ are all invertible.

Let $t \in \Mor\C(A, B)$ be a $\S^{-1}$-weak equivalence. Then for any $D \inn \P$, 
	\begin{equation}
	\H\colon \SW\C \to \Ab, \quad (A, n) \mapsto \Mor\P(D, R(A, n))
	\end{equation}
is a homology theory by \cite[Theorem 2.3.8]{andreasthomE} and every $s \in \S$ is an $\H$-equivalence, hence we see that $t$ too is an $\H$-equivalence. By Yoneda's lemma, $R(\Omega^{0}t)$ is invertible in $\P$. Thus $R$ induces a functor $R_*\colon\Ho{\R_{\S}\C} \to \P$ which is easily seen to intertwine the $\Omega$'s, hence induces a functor $\widehat{R}\colon \SW{\R_{\S}\C} \to \P$. Since $R$ is a triangulated homology theory, $\widehat{R}$ is a triangulated functor and $R = \widehat{R} \circ Q$. The uniqueness of $\widehat{R}$ is clear.
\end{proof}

In other words, $\SW{\R_{\S}\C}$ is the universal triangulated homology theory for which all morphisms of $\S$ are equivalences (cf.\ \cite[Definition 2.3.3]{andreasthomE}).

\section{Applications to the Category of \cast-algebras}\label{sec cast}

Let $\Calg$ denote the category of {\em \cast-algebras} and {\em $*$-homo\-morphisms}.
It is complete and cocomplete and pointed -- the zero object is the zero algebra $0$ -- symmetric monoidal category with respect to the {\em maximal} tensor product, which we denote by $\otimes$ (instead of the more standard notation $\otimes_{\max}$, since we will not consider any other tensor product). We refer to \cite{MR2513331} for the details.

The category $\Calg$ is naturally enriched over $\Top$, the Cartesian closed category of compactly generated weakly Hausdorff topological spaces. Indeed, since \cast-algebras are normed, they are compactly generated and weakly  Hausdorff as spaces, hence there is a forgetgul functor $\Calg \to \Top$. For \cast-algebras $A$ and $B$, we give $\Mor\Calg(A, B)$ the subspace topology from $\Mor\Top(A, B)$ via the forgetful functor. It is easy to see that $\Mor\Calg(A, B)$ is a closed subspace of $\Mor\Top(A, B)$, hence itself a compactly generated weakly Hausdorff space. 

Let $\Aalg \subset \Calg$ denote the full subcategory of {\em abelian} \cast-algebras. By the Gelfand-Naimark duality, $\Aalg$ is equivalent to the opposite category of the category $\CH_\pt$ of pointed, compact Hausdorff topological spaces and pointed continuous maps. If $X$ is a compact Hausdorff space, we write $C(X)$ for the (unital) \cast-algebra of continuous functions on $X$. If in addition $X$ has a base point, we write $C_0(X)$ for the \cast-algebra of continuous functions on $X$ vanishing at the base point. 

\begin{rem} The category $\Calg$ of \cast-algebras is also enriched over the category of Hausdorff spaces, using the compact-open topology on morphism spaces. However, in order to facilitate the connection to algebraic topology, we use the compactly generated compact-open topology. Note that if $A$ is separable, then the compact-open topology on $\Mor\Calg(A, B)$ is metrizable, hence compactly generated.
\end{rem}

\begin{lem}\label{lem B(X)} Let $B$ be a \cast-algebra and let $X$ be a compact Hausdorff space. Then the set of maps $\Mor\Top(X, B)$ is naturally a \cast-algebra isomorphic to $C(X) \otimes B$.
\end{lem}
\begin{proof} By \cite[Proposition 2.13]{cgwh} the topology on $\Mor\Top(X, B)$ coincides with the topology given by the norm $||f|| \coloneqq \sup_{x \in X}||f(x)||_B$. The rest is standard (cf.\ \cite[Corollary T.6.17]{MR1222415}). 
\end{proof}

\begin{notation}\label{notate BX} Let $B$ be a \cast-algebra and let $X$ be a compact Hausdorff space. We write $B^{X}$ for the \cast-algebra $\Mor\Top(X, B) \cong C(X) \otimes B$. For $x \in X$, the {\em evaluation} map $f \mapsto f(x)$ is denoted $\ev_{x}\colon B^{X} \to B$.
\end{notation}

The following is the main property of the enrichment that we use. See also \cite[Proposition 3.4]{MR2222510} and \cite[Proposition 24]{MR2513331}.
\begin{lem}\label{lem enrichment} Let $A$ and $B$ be \cast-algebras and let $X$ be a compact Hausdorff space. Then there is an identification
	\begin{equation}\label{eq:left adjoint of C(X)}
	\Mor{\Top}(X, \Mor{\Calg}(A, B)) \cong \Mor{\Calg}(A, B^{X})
	\end{equation}
natural in $A$, $B$ and $X$.	 
\end{lem}
\begin{proof} Since $A$ and $B$ are compactly generated weakly Hausdorff spaces, we have a natural identification
	\begin{equation}
	\Mor{\Top}(X, \Mor{\Top}(A, B)) \cong \Mor{\Top}(A, \Mor{\Top}(X, B)),
	\end{equation}
by \cite[Proposition 2.12]{cgwh}. Hence by Lemma~\ref{lem B(X)} 
	\begin{equation}
	\Mor{\Top}(X, \Mor{\Top}(A, B)) \cong \Mor{\Top}(A, B^{X}).
	\end{equation}
Now it is easy to check that this restricts to the identification in (\ref{eq:left adjoint of C(X)}).

\end{proof}

Often we will make this identification implicitly. 

\begin{rem}\label{rem pointed enr} Lemma~\ref{lem B(X)} and Lemma~\ref{lem enrichment} have pointed analogues.

Let $\Top_{*}$ denote the category of pointed spaces and pointed maps. Since \cast-algebras have a natural base point $0$ and $*$-homomorphisms are pointed maps, there is in fact a forgetful functor $\Calg \to \Top_{*}$ and $\Calg$ is enriched over $\Top_{*}$.
 
Let $A$ and $B$ are \cast-algebras and let $X$ be a pointed compact Hausdorff space. Let $B^{X}$ denote the \cast-algebra $C_{0}(X) \otimes B \cong \Mor\Top_{*}(X, B)$. Then it follows from Lemma~\ref{lem enrichment} that there is a natural identification 
	\begin{equation}\label{eq:left adjoint of C0(X)}
	\Mor{\Top_{*}}(X, \Mor{\Calg}(A, B)) \cong \Mor{\Calg}(A, B^{X}).
	\end{equation}
\end{rem}

\begin{cor}\label{cor pullback} For any $D \inn \Calg$, the functor $\Mor\Calg(D, -)\colon \Calg \to \Top$ preserves pullbacks.
\end{cor}
\begin{proof} Let $D$ be fixed and let $F \coloneqq \Mor\Calg(D, -)$.

Consider a pullback diagram 
	\begin{equation}
	\xymatrix{
	A \times_B E \ar[r] \ar[d] & E \ar[d]\\
	A \ar[r] & B
	}
	\end{equation}
in $\Calg$. We need to prove that the natural map 
	\begin{equation}
	\Phi\colon F(A \times_B E) \to F(A) \times_{F(B)} F(E)
	\end{equation}
is a homeomorphism. It is clear that $\Phi$ is a continuous bijection. Hence it suffices to prove that for any $X$ compact Hausdorff, a map $X \to F(A \times_B E)$ is continuous if the compositions $X \to F(A)$ and $X \to F(E)$ are continuous. However, this follows from Lemma~\ref{lem enrichment} and its proof.
\end{proof}

\subsection{Ordinary Homotopy Theory}

\begin{notation} We denote the interval $[0, 1] \coloneqq \{x \in \R \mid 0 \le x \le 1\}$ by $I$.
\end{notation}

\begin{defn}\label{defn homot} Let $A$ and $B$ be \cast-algebras. Two $\ast$-homomorphisms $f_{0}$, $f_{1} \colon A \to B$ are said to be {\em homotopic} if there exists a $\ast$-homomorphism $F \colon A \to B^{I}$, called a homotopy, such that $f_{0} = \ev_{0} \circ F$ and $f_{1} = \ev_{1} \circ F$, where $\ev_{t} \colon B^{I} \to B$ is the evaluation map at $t \in [0, 1]$. We denote the set of homotopy classes of $*$-homomorphisms $A \to B$ by 
	\begin{equation}
	[A, B] \coloneqq \{\text{homotopy classes of maps $A \to B$}\}.
	\end{equation}
\end{defn}

\begin{rem}\label{rem homotopy} By Lemma~\ref{lem enrichment}, two $*$-homomorphisms $f_{0}$, $f_{1} \colon A \to B$ are homotopic if and only if $\pi_{0}(f_{0}) = \pi_{0}(f_{1})$ in $\pi_{0}(\Mor\Calg(A, B))$, where $\pi_{0}$ is the path-connected components functor.
\end{rem}

The (ordinary) {\em homotopy category} of \cast-algebras is the category of \cast-algebras and homotopy classes of $*$-homomorphisms. In view of Remark~\ref{rem homotopy}, we denote this category $\pi_{0}\Calg$:
	\begin{equation}
	\Mor{\pi_{0}\Calg}(A, B) \coloneqq \pi_{0}\Mor\Calg(A, B) = [A, B]
	\end{equation}

 We have a natural functor $\pi_{0} \colon \Calg \to \pi_{0}\Calg$.

We now give $\Calg$ the structure of a category of fibrant objects, whose homotopy category is $\pi_{0}\Calg$, following \cite{MR757510}. We consider $\Top$ as a category of fibrant objects using weak homotopy equivalences and Serre fibrations (see Example~\ref{ex model}) and we ``pullback'' this structure to $\Calg$ using Corollary~\ref{cor pullback}. 	  

\begin{defn}\label{defn homotopy equivalence}
A $*$-homomorphism $t \in \Calg$ is called a {\em homotopy equivalence} if $\pi_{0}(t)$ is invertible in $\pi_{0}\Calg$.
\end{defn}

\begin{lem}\label{lem homotopic otimes f} Let $F \inn \Calg$. If $f_{0}$, $f_{1} \in \Mor\Calg(A, B)$ are homotopic, then the maps $f_{0} \otimes \id_{F}$, $f_{1} \otimes \id_{F}\in \Mor\Calg(A \otimes F, B \otimes F)$ are homotopic. In particular, the functor $A \mapsto A \otimes F$ preserves homotopy equivalences.
\end{lem}
\begin{proof} Clear.
\end{proof}

\begin{lem}\label{lem homotopic} If $f_{0}$, $f_{1} \in \Mor\Calg(A, B)$ are homotopic then for any $D \inn \Calg$, the induced maps $(f_{0})_{*}$, $(f_{1})_{*}\colon \Mor\Calg(D, A) \to \Mor\Calg(D, B)$ are homotopic in $\Top$. 
\end{lem}
\begin{proof} Follows from Lemma~\ref{lem enrichment}.
\end{proof}

\begin{prop}\label{prop he} Let $t \in \Mor\Calg(A, B)$. Then $t$ is a homotopy equivalence if and only if the induced map 
	\begin{equation}
	t_{*}\colon \Mor\Calg(D, A) \to \Mor\Calg(D, B)
	\end{equation}
is a weak homotopy equivalence in $\Top$ for all $D \inn \Calg$.
\end{prop}
\begin{proof} If $t \in \Mor\Calg(A, B)$ is a homotopy equivalence, then for any $D \inn \Calg$, the induced map $t_{*} \colon \Mor\Calg(D, A) \to \Mor\Calg(D, B)$ is a homotopy equivalence by Lemma~\ref{lem homotopic}, hence a weak homotopy equivalence. Conversely, suppose that the induced map $t_{*}\colon \Mor\Calg(D, A) \to \Mor\Calg(D, B)$ is a weak homotopy equivalence in $\Top$ for all $D \inn \Calg$. Then in particular, $\pi_{0}(t)_{*} = \pi_{0}(t_{*}) \colon \pi_{0}\Mor\Calg(D, A) \to \pi_{0}\Mor\Calg(D, B)$ is a bijection for all $D \inn \Calg$. By Yoneda's Lemma, $\pi_{0}(t)$ is invertible.
\end{proof}

%
%
%
%
%
\begin{defn}\label{defn fibration} A $*$-homomorphism $p\colon E \to B$ is called a {\em Schochet fibration} if the induced map
	\begin{equation}
	p_{*}\colon \Mor\Calg(D, E) \to \Mor\Calg(D, B)
	\end{equation}
has the path lifting property (i.e.\ the right lifting property with respect to the inclusion $\{0\} \hookrightarrow [0, 1]$) in $\Top$ for all $D \inn \Calg$.
\end{defn}

\begin{defn} Let $f \colon A \to B$ be a $*$-homomorphism. Let $Nf$ denote the pullback
	\begin{equation}
	\xymatrix{NF \ar[r] \ar[d] &B^{I} \ar[d]^{\ev_{0}}\\
	A \ar[r]^{f} & B 
	}
	\end{equation}
\end{defn}

\begin{lem}\label{lem fib splitting} A $*$-homomorphism $p \colon E \to B$ is a Schochet fibration if and only if the natural map $E^{I} \to Np$ splits.
\end{lem}
\begin{proof} See \cite[Proposition 1.10]{MR757510}.
\end{proof}

\begin{lem}\label{lem fib otimes f} For any $F \inn \Calg$, the functor $A \mapsto A \otimes F$ preserves pullbacks and Schochet fibrations.
\end{lem}
\begin{proof} The functor $A \mapsto A \otimes F$ preserves pullbacks by \cite[Remark 3.10]{MR1716199} and it preserves Schochet fibrations by Lemma~\ref{lem fib splitting}. See \cite[Proposition 1.11]{MR757510}.
\end{proof}

\begin{prop}\label{prop sch fib} Let $p \in \Mor\Calg(E, B)$. Then $p$ is a Schochet fibration if and only if the induced map
	\begin{equation}\label{eq ind map fib}
	p_{*}\colon \Mor\Calg(D, E) \to \Mor\Calg(D, B)
	\end{equation}
is a Serre fibration (i.e.\ has the right lifting property with respect to the natural inclusion $\{0\} \times I^{n} \hookrightarrow [0, 1] \times I^{n}$ for all $n \ge 0$) in $\Top$ for all $D \inn \Calg$. 
\end{prop}
\begin{proof} Clearly, Serre fibrations have the path lifting property. Hence it is enough to show that if $p$ is a Schochet fibration then $p_{*}$ is a Serre fibration. For any compact Hausdorff space $X$, by Lemma~\ref{lem enrichment}, the map $p_{*}\colon \Mor\Calg(D, E) \to \Mor\Calg(D, B)$ has the right lifting property with respect to $\{0\} \times X \hookrightarrow [0, 1] \times X$ if and only if the map $(\id_{C(X)} \otimes p)_{*} \colon \Mor\Calg(D, C(X) \otimes E)  \to \Mor\Calg(D, C(X) \otimes B)$ has the path lifting property. Now the proof is complete by Lemma~\ref{lem fib otimes f}.
\end{proof}

The following theorem is contained in \cite{MR757510}. 
\begin{thm}\label{thm ord homotopy} The category of \cast-algebras $\Calg$ is a pointed category of fibrant objects with weak equivalences the homotopy equivalences and fibrations the Schochet fibrations, whose homotopy category is the ordinary homotopy category i.e.\ $\Ho\Calg = \pi_{0}\Calg$.
\end{thm}
\begin{proof} 
Properties (F0), (F1), (F3) follow from Proposition~\ref{prop sch fib}. Properties (W1) and (W2) are clear (or use Proposition~\ref{prop he}). For properties (F2) and (FW1), use Corollary~\ref{cor pullback} in addition.

For (FW2): Let $[a, b]$ be a compact interval, $a < b$,  and let 
	\begin{equation}
	B^{[a, b]} \coloneqq \Mor{\Top}([a, b], B) \cong C[a, b] \otimes B.	
	\end{equation}
and let $\ev_{t}\colon B^{[a, b]} \to B$ denote the evaluation at $t \in [a, b]$ (cf.\ Notation~\ref{notate BX}). Then the map $(\ev_{a}, \ev_{b})\colon B^{[a, b]} \to B\times B$ is a Schochet fibration, since the rectangle $[0, 1]\times [a, b]$ retracts to the union of its three sides $\sqsubset$. The constant-path map $s\colon B \to B^{[a, b]}$ is a homotopy equivalence with homotopy inverse $\ev_{a}$. Thus $(B^{[a, b]}, s, \ev_a, \ev_b)$ is a path-object for $B$. For fixed $a$ and $b$, this is functorial.

It follows from the construction of the path-object in $\Calg$ that two $*$-homo\-morphisms are right-homotopic if and only if they are homotopic in the sense of Definition~\ref{defn homot} and this happens if and only if they are homotopic in the sense of Definition~\ref{Brown homotopy}.
Hence 
	\begin{equation}
	\Ho{\Calg} = \pi\Calg = \pi_0\Calg.
	\end{equation}
\end{proof}

Note that $\Calg$ has a functorial path-object, given by $C[0, 1] \otimes B = B^{I}$, hence also a functorial loop-object $\Omega B \coloneqq C_0(0, 1) \otimes B$.

\begin{rem}\label{rem naming}
Schochet called these maps {\em cofibrations} in \cite{MR757510}, because, under the Gelfand-Naimark duality, the condition in Definition~\ref{defn fibration} for a $*$-homomorphism of abelian algebras corresponds to the homotopy extension property for the corresponding map of (pointed compact Hausdorff) spaces.


In a similar way, it is customary that $\Mor{\Top_{*}}(S^{1}, B) \cong C_{0}(S^{1}) \otimes B \cong C_{0}(0, 1) \otimes B$ is called the {\em suspension} of $B$, since $C_{0}(S^{1}) \otimes C_{0}(X) \cong C_{0}(S^{1} \wedge X)$ for $B = C_{0}(X)$, where $X$ is a pointed compact Hausdorff space. See also Remark \ref{rem abelian algebra}.

However, for the sake of consistency, in this paper we will keep our notations and terminologies compatible with that of Section 1. 
\end{rem}

The stable homotopy category $\SW{\Calg}$ is the {\em suspension-stable homotopy category of \cast-algebras} studied by Rosenberg \cite{MR658514} and Schochet \cite{MR757510}. 

\begin{rem} Let $\Csep$ denote the category of {\em separable} \cast-algebras. Then considering only $D$ separable in Definitions \ref{defn homotopy equivalence} and \ref{defn fibration}, we get a structure of a category of fibrant objects on $\Csep$.
\end{rem}

\begin{lem}\label{lem Schochet fibration surjective} All Schochet fibrations are surjective. 
\end{lem}
\begin{proof}
Let $p\colon E \fib B$ be a Schochet fibration. Consider the universal algebra generated by a positive contraction:
	\begin{equation}
	C \coloneqq \Calg(x \mid 0 \le x \le 1) = C_0(0, 1].
	\end{equation}
Then for any $b \in B$, $0 \le b \le 1$,  there is a path 
	\begin{equation}
	[0, 1] \ni r \mapsto (x \mapsto rb) \in \Mor\Calg(C, B),
	\end{equation}
which lifts to $0 \in \Mor\Calg(C, E)$ at $r=0$. Lifting the path to $\Mor\Calg(C, E)$, we get $e \in E$, $0 \le e \le 1$, such that $p(e) = b$. It follows that $p$ is surjective.
\end{proof}

\begin{rem}\label{rem prod coprod} The following are well-known and/or easy to see.
\begin{enumerate}
\item\label{item prod coprod} The localization $\Calg \to \Ho\Calg$ preserves arbitrary coproducts and arbitrary products:
	\begin{align}
	\MorHo{\Calg}{\coprod_{i \in \Lambda} A_i, B} &\cong \prod_{i \in \Lambda}\MorHo{\Calg}{A_i, B},\\
	\MorHo{\Calg}{A, \prod_{i \in \Lambda} B_{i}} &\cong \prod_{i \in \Lambda}\MorHo{\Calg}{A, B_{i}}.
	\end{align}
\item The loop functor $\Omega\colon \Ho\Calg \to \Ho\Calg$ preserves finite products:
	\begin{equation}
	\Omega (B_{1} \times B_{2}) \cong \Omega B_{1} \times \Omega B_{2},
	\end{equation}
but not finite coproducts (for example, the natural map $\Omega \Co \coprod \Omega \Co \to \Omega (\Co \coprod \Co)$ is not a homotopy equivalence). 

\item\label{item loop inf-prod} The loop functor $\Omega\colon
  \Ho\Calg \to \Ho\Calg$ does {\em not} preserve infinite products,
and in particular does not admit a left-adjoint; see Appendix~\ref{nomodel}. 


\item The ``stable homotopy functor'' $\Omega^{0}\colon \Ho\Calg \to \SW\Calg$ preserves finite products, but not finite coproducts.
\end{enumerate}

\end{rem}

\subsection{\cast-stable Homotopy Theory}

Let $\Compact$ denote the \cast-algebra of compact operators on a separable infinite-dimensional Hilbert space.

\begin{prop} Defining the weak equivalences to be
	\begin{equation}
	\left\{t \in \Calg \mid t \otimes \id_{\Compact} \text{ is a homotopy equivalence}\right\}
	\end{equation}
and the fibrations to be
	\begin{equation}
	\left\{p \in \Calg \mid p \otimes \id_{\Compact} \text{ is a Schochet fibration}\right\}
	\end{equation}  defines a category of fibrant objects on $\Calg$, denoted $\M$.
\end{prop}
\begin{proof} 
This is clear since $-\otimes \id_{\Compact}$ preserves pullbacks.
\end{proof}

Let $e_{11}\colon \Co \to \Compact$ denote a rank-one projection. Then 
for any $B \inn \M$, the morphism $\id_{B} \otimes e_{11}$ is a weak equivalence in $\M$. It follows that $\Ho{\M}$ is the ``monoidal" localization $\Ho\Calg[ \otimes e_{11}^{-1}]$: 
	\begin{align}
	\MorHo{\M}{A, B} &\cong \MorHo{\Calg}{A, B \otimes \Compact}\\ 
	&\cong \MorHo{\Calg}{A \otimes \Compact, B \otimes \Compact}.
	\end{align}
In the notation of \cite{MR1068250}, the categories $\Ho{\M}$ and $\SW{\M}$ are the not necessarily separable versions of $\cat{TH}$ and $\cat{TS}$ respectively. When restricted to the abelian algebras, $\SW{\M}$ gives the $kk$ groups of D{\u{a}}d{\u{a}}rlat-McClure \cite{MR1800209}.

\subsection{Topological $K$-Theory}\label{subsec K}



Taking $\H$ to be topological $K$-theory in Theorem~\ref{thm fibre homology fibrant object}, we get a category $\K = R_{K}\Calg$ of fibrant objects whose weak equivalences are $K$-equi{\-}valences and fibrations are Schochet fibrations. Compare \cite{MR2222510} and \cite{MR2193334}. It follows from Theorem~\ref{thm UCT}, that $\Ho\K$ has small hom sets. 

Let $\Compact$ be the algebra of compact operators on a separable Hilbert space and let $e_{11}\colon \Co \to \Compact$ denote a rank-one projection. Then 
	\begin{equation}\label{eq otimes e}
	\id_{A} \otimes e_{11}\colon A \to A \otimes \Compact
	\end{equation}
is a $K$-equivalence. We also have a natural isomorphism $\Omega^{2} A \to A \otimes \Compact$ in $\Ho\K$, since Bott periodicity can be implemented by a boundary map associated to a Toeplitz type extension.  It follows that 
	\begin{equation}
	\Omega\colon \MorHo{\K}{A, B} \to \MorHo{\K}{\Omega A, \Omega B}
	\end{equation}
is invertible. Hence $\K$ is stable and the natural functor $\Ho\K \to \SW\K$ is an equivalence of categories. In particular, $\Ho\K$ is a triangulated category in a natural way, and $\SW\Calg \to \Ho\K$ is a triangulated functor.

%
%
%

The following is a version of the Universal Coefficient Theorem of Rosenberg and Schochet (cf.\ \cite{MR894590}). It can be deduced from results in \cite{MR2193334}, but we give a self-contained proof.
\begin{thm}\label{thm UCT} 
For $B \inn \K$, we have
	\begin{align}
	\label{HoK C}\MorHo{\K}{\Co, B} &\cong K_0(B).
	\end{align}
More generally, for $A$, $B \inn \K$, there is a natural short exact sequence
	\begin{equation}\label{eq UCT}
	\xymatrix{
	\Ext(K_{*+1}(A), K_*(B)) \ar@{>->}[r] & \MorHo{\K}{A, B} \ar@{->>}[r] & \Hom(K_*(A), K_*(B))
	},
	\end{equation}
where
	\begin{align}
	\Hom(K_*(A), K_*(B)) &\coloneqq \bigoplus_{i = 0, 1} \Hom_\Z(K_i(A), K_i(B)) \quad\text{and}\\
	\Ext(K_{*-1}(A), K_*(B)) &\coloneqq \bigoplus_{i = 0, 1} \Ext^1_\Z(K_{i-1}(A), K_i(B)).
	\end{align}
\end{thm}
\begin{proof} We have a natural (additive) map 
	\begin{equation}\label{eq uct map}
	\MorHo{\K}{A, B} \to \Hom_\Z(K_*(A), K_*(B)).
	\end{equation}
We claim that this is an isomorphism if $K_{*}(A)$ is free --- for $A = \Co$ we get (\ref{HoK C}). 

Indeed, suppose that $K_{*}(A)$ is free. First recall that we have natural isomorphisms 
	\begin{align}
	K_0(D) &= \MorHo{\Calg}{q\Co, D \otimes \Compact},\\
	K_1(D) &= \MorHo{\Calg}{\Omega\Co, D \otimes \Compact},
	\end{align}
where $q\Co$ is the kernel of the folding map $(\Co \coprod \Co \to \Co)$. We have a $K$-equivalence $q\Co \we \Co$.

Then it is clear that any map $K_{*}(A) \to K_{*}(B)$ can be implemented by an element of the form
	\begin{equation}
	\xymatrix{
	\left(\coprod_{I} q\Co\right) \coprod \left(\coprod_{J} \Omega \Co\right) \ar[d]^{\wr} \ar[r]& B \otimes \Compact\\
	A \otimes \Compact  & \\
	A \ar[u]_{\wr} &B \ar[uu]_{\wr}
	}
	\end{equation}
in $\Ho\K$. Hence (\ref{eq uct map}) is surjective. To see injectivity of (\ref{eq uct map}), let 
	\begin{equation}\label{eq element of HoK}
	\xymatrix{
	A & A' \ar[r] \ar[l]_{\sim}& B}
	\end{equation}
be a morphism in $\MorHo{\K}{A, B}$ that maps to $0 \in \Hom(K_*(A), K_*(B))$. Then we can complete (\ref{eq element of HoK}) to a homotopy-commutative diagram
	\begin{equation}
	\xymatrix{
	\left(\coprod_{I} q\Co\right) \coprod \left(\coprod_{J} \Omega \Co\right) \ar[d]^{\wr} \ar[r]^-{\sim}& A'\otimes \Compact \ar[r]& B\otimes \Compact\\
	A \otimes \Compact &  &\\
	A\ar[u]_{\wr} &A' \ar[r] \ar[l]_{\sim} \ar[uu]_{\wr}&B \ar[uu]_{\wr}
	}
	\end{equation}
in $\Ho\Calg$. Then the top horizontal map is null-homotopic, i.e.\ zero in $\Ho\Calg$, hence zero in $\Ho\K$. In other words, (\ref{eq uct map}) is injective if $K_{*}(A)$ is free. 

The general case follows using a geometric resolution of $K_{*}(A)$. See for instance \cite{math/0403511}.
\end{proof}

\subsection{$K\!K$-Theory}\label{subsec KK}
In the next two subsections, we will concentrate on the category $\Csep$ of {\em separable} \cast-algebras. We refer to \cite[Chapter VIII]{MR1656031} for details about Kasparov's $K\!K$-theory.

Recall that, in the Cuntz picture of $K\!K$-theory (cf.\ \cite{MR899916}), we have
	\begin{equation}
	K\!K(A, B) \coloneqq \MorHo{\Calg}{qA, B \otimes \Compact} = \MorHo{\M}{qA, B},
	\end{equation}
where $qA$ is the kernel of the folding map $\id_A \coprod \id_A\colon A \coprod A \to A$.

\begin{lem}\label{lem les for kk fib}
Let $E \to B$ be a Schochet fibration with fibre $F$. Then for any $D \in \Calg$, we have a natural 6-term exact sequence:
	\begin{equation}
	\xymatrix{
	K\!K(D, F) \ar[r] & K\!K(D, E) \ar[r] & K\!K(D, B) \ar[d]\\
	K\!K(D, \Omega B) \ar[u] & \ar[l] K\!K(D, \Omega E) & \ar[l] K\!K(D, \Omega F)
	}
	\end{equation}
\end{lem}
\begin{proof} Follows from the fibre exact sequence (cf.\ Theorem~\ref{thm Puppe}) in $\Ho\M$ (or $\Ho\Calg$) and Bott Periodicity.
\end{proof}
	
\begin{defn}
A $*$-homomorphisms $t\colon A \to B$ in $\Csep$ is called a $K\!K$-equivalence if 
	\begin{equation}
	t_{*}\colon K\!K(D, A) \to K\!K(D,B)
	\end{equation}
is an isomorphism for all $D \inn \Csep$. 
\end{defn}

The following example is the cornerstone of the Cuntz picture of $K\!K$-theory. 
\begin{ex}
For any $A \inn \Csep$, the composition
	\begin{equation}
	\xymatrix{qA \ar@{>->}[r] & A \coprod A \ar[r]^-{\id_{A} \coprod 0} & A}
	\end{equation}
is a $K\!K$-equivalence. 
\end{ex}

In particular, we have an identification
	\begin{equation}
	K\!K(A, B) \cong \MorHo{\M}{qA, qB} = \MorHo{\Calg}{qA \otimes \Compact, qB \otimes \Compact}.
	\end{equation}
Under this identification, composition of $K\!K$-elements correspond to composition of homotopy classes (cf.\ \cite{MR899916}). In particular, a $*$-homomorphism is a $K\!K$-equivalence if and only if it determines an invertible element in $K\!K$, as expected.
	
\begin{thm}\label{thm kk fibrant} The category of {\em separable} \cast-algebras forms a category of fibrant objects with weak equivalences the $K\!K$-equivalences and fibrations the Schochet fibrations, denoted $\KK$, whose homotopy category $\Ho\KK$ is equivalent to the $K\!K$-category of Kasparov. It follows that Kasparov's $K\!K$-category is a stable triangulated category.
\end{thm}
\begin{proof} The category of fibrant objects structure follows from Theorem~\ref{thm fibre homology fibrant object}, since $K\!K(D, -)$ gives a homology theory on $\Calg$ in the sense of Definition~\ref{defn fht} for all $D$ by Lemma~\ref{lem les for kk fib}.

Now the functor $\Ho\KK \to K\!K$ given by $A \mapsto qA \otimes \Compact$ is easily seen to be an equivalence of categories. Stability follows from Bott Periodicity.
\end{proof}

\begin{rem} Note that in Theorem \ref{thm kk fibrant}, we can take the semi-split surjections, i.e.\ surjections with a completely positive contractive splitting, to be the fibrations. Indeed, the only nontrivial part is (FW1): if $p\colon E \to B$ is a semi-split surjection which is also a $K\!K$-equivalence and $f\colon A \to B$ is arbitrary, then the pullback $f^*(p)$ is also a $K\!K$-equivalence. However, this is clear since if $p$ is a semi-split surjection with kernel $F$, then $F \to \Cone{p}$ is a $K\!K$-equivalence (see \cite[Theorem 19.5.5]{MR1656031}), hence $p$ is a $K\!K$-equivalence if and only $F$ is $K\!K$-contractible if and only if $f^*(p)$ is a $K\!K$-equivalence (see Diagram (\ref{eq pullback of fibration seq})).

Note also that Schochet fibrations and semi-split surjections give rise to the same class of distinguished triangles in $\Ho\KK \cong \SW\KK$.
\end{rem}

\subsection{Universal Homology Theories}\label{subsec E}

We consider $\Csep$ as a category of fibrant objects with weak equivalences the homotopy equivalences and fibrations the Schochet fibrations. In this subsection, we identify various localizations of $\Csep$.

\begin{defn}
A {\em fibre homology theory} on $\Csep$ is a homology theory the pointed category of fibrant objects $\Csep$ in the sense of Definition~\ref{defn fht} i.e.\ a homological functor on the triangulated category $\SW\Csep$ to $\Ab$ 
\end{defn}

\begin{defn}
We say that a fibre homology theory $\H$ on $\Csep$ is {\em excisive} with respect to a surjection $p$, if the inclusion $\ker(p) \to \Cone{p}$ is an $\H$-equivalence. A {\em homology theory} on $\Csep$ is a fibre homology theory excisive with respect to all surjections.
\end{defn}

\begin{defn}\label{defn we} We say that a morphism $t \in \Csep$ is a {\em weak equivalence} if it is an $\H$-equivalence for all homology theories $\H$ on $\Csep$.
\end{defn}

\begin{rem} Note that homotopy equivalences are weak equivalences.
\end{rem}

\begin{thm}\label{thm AM} The category $\Csep$ forms a pointed category of fibrant objects with weak equivalences as in Definition \ref{defn we} and fibrations the Schochet fibrations, whose stable homotopy category is triangulated equivalent to the stable homotopy category of \cite[Theorem 3.3.5]{andreasthomE}.
\end{thm}

By \cite{MR1262931}, the stable homotopy category mentioned above is equivalent to the suspension-stable version of the strong shape category.

\begin{proof}
It follows from Theorem~\ref{thm localization} that the stable homotopy category is a {\em universal} triangulated homology theory in the sense of \cite[Definition 2.3.3]{andreasthomE}. Then \cite[Theorem 3.3.6]{andreasthomE} finishes the proof.

\end{proof}

For a Hilbert space $H$, let $e_{H}\colon \Co \to \Compact(H)$ denote a rank-one projection.
\begin{defn} A (fibre) homology theory $\H$ on $\Csep$ is said to be 
\begin{enumerate}
\item {\em matrix-invariant} if $\id_B \otimes e_{H}$ is an $\H$-equivalence for all $B \inn \Csep$ and $H$ finite dimensional and 
\item {\em \cast-invariant} if $\id_B \otimes e_{H}$ is an $\H$-equivalence for all $B \inn \Csep$ and $H$ separable.
\end{enumerate}
\end{defn}

\begin{defn} A morphism $t \in \Csep$ is said to be 
\begin{enumerate}
\item an {\em $\bu$-equivalence} if it induces isomorphism on all matrix-invariant homology theories and
\item an {\em $E$-equivalence} if it induces isomorphism on all \cast-invariant homology theories.
\end{enumerate}
\end{defn}

\begin{thm}\label{thm bu E} 
\begin{enumerate}
\item The category $\Csep$ forms a pointed category of fibrant objects with weak equivalences the $\bu$-equivalences and fibrations the Schochet fibrations, whose stable homotopy category is triangulated equivalent to the category $\bu$ of \cite[Theorem 4.2.1]{andreasthomE}.
\item The category $\Csep$ forms a stable pointed category of fibrant objects with weak equivalences the $E$-equivalences and fibrations the Schochet fibrations, whose homotopy category is a triangulated category, equivalent to the $E$-theory of Higson.
\end{enumerate}
\end{thm}
\begin{proof} Follows from Theorem~\ref{thm localization} and the universal properties of $\bu$ and $E$ (cf. \cite{andreasthomE}). 
\end{proof}


\begin{rem}
\begin{enumerate}
\item Let $p\colon E \to B$ be a surjection with kernel $F$. Then $p$ is a weak equivalence in the sense of Definition~\ref{defn we} if and only if $F$ is $\H$-acyclic for all homology theories $\H$ on $\Csep$. Indeed, we have a map of extensions where the vertical maps are all weak equivalences:
	\begin{equation}\label{diag FEB FpNpB}
	\xymatrix{
	0  \ar[r]& F\ar[d]^-{\wr} \ar[r]^{i}& E \ar[r]^{p}\ar[d]^-{\wr} & B \ar@{=}[d]  \ar[r] & 0\\
	0 \ar[r]& \Cone{p} \ar@{>->}[r]& \Cyl{p} \ar@{->>}[r]& B \ar[r] & 0
	}.
	\end{equation}
Hence the claim follows from the naturality of the long exact sequence associated to homology theories. It follows that in Theorem~\ref{thm AM} and Theorem~\ref{thm bu E}, we can take the fibrations to be all surjections.	However, the distinguished triangles in the stable homotopy category would be the same (see the diagram (\ref{diag FEB FpNpB})). 
\item We can also describe the $\KK$-category of Kasparov as the universal split-exact triangulated homology theory in a similar way.
\end{enumerate}
\end{rem}

\appendix
\section{No Quillen Model Structure\\ (following Andersen-Grodal)}\label{nomodel}

As noted in the introduction, the homotopy theory of \cast-algebras
does not come from a Quillen model structure. This was perhaps first
pointed out as part of a 1997 preprint by Andersen-Grodal
\cite{grodal-fib},  where they also established a {\em Baues fibration
  category structure} \cite{baues89} on $C^*$-algebras (a notion very
similar to a category of fibrant objects; see \cite[Rem.~I.1a.6]{baues89}). Since their
work however remains unpublished,
we, by permission of the authors, reproduce their non-existence argument in this
appendix.


Recall that if $\M$ is a Quilen model category, then the full subcategory $\M_f$ of fibrant objects in $\M$ is a category of fibrant objects (cf.\ Example \ref{ex model}). 
\begin{thm}\label{thm no Quillen} Let $\Calg$ denote the pointed category of fibrant objects of Theorem~\ref{thm ord homotopy}. Then $\Calg$ is not the full subcategory of fibrant objects of a Quillen model category.
\end{thm}

The essential part of the proof is to see that the loop functor
does not admit a left adjoint, as already remarked on in Remark~\ref{rem prod coprod}(\ref{item loop inf-prod}).

\begin{lem}\label{lem omega adjoint}
Let $\M_{f}$ be the full subcategory fibrant objects of a Quillen model category $\M$, considered as a category of fibrant objects as in Example~\ref{ex model}. Then the loop-functor 
	\begin{equation}
	\Omega\colon \Ho{\M_{f}} \to \Ho{\M_{f}}
	\end{equation}
admits a left-adjoint.
\end{lem}
\begin{proof} Follows from Theorem I.1.1 and Theorem I.2.2 of \cite{MR0223432} and the definitions.
\end{proof}

The following Lemma is clear.
\begin{lem}\label{rem abelian algebra}
Let $\Aalg \subseteq \Calg$ denote the full subcategory consisting of abelian \cast-algebras. Then $\Aalg$ is a reflective (monoidal) subcategory of $\Calg$ -- the left-adjoint of the inclusion $i\colon \Aalg \to \Calg$ is the abelianization $(-)^{\mathrm{ab}}\colon \Calg \to \Aalg$:
	\begin{equation}\label{eq adjunction of A in C}
	\Mor\Aalg(D^{\mathrm{ab}}, B) \cong \Mor{\Calg}(D, iB),   
	\end{equation}
for $D \inn \Calg,\, B \inn \Aalg$. 
\qed
\end{lem}

In particular, $\Aalg$ is a pointed category of fibrant objects (cf.\ Example~\ref{ex ref subcat}).
\begin{cor} 
 The homotopy category $\Ho\Aalg$ is a full reflective subcategory of $\Ho\Calg$ and the loop-functor 
 	\begin{equation}
	\Omega\colon \Ho\Aalg \to \Ho\Aalg	
	\end{equation}
is the restriction of $\Omega\colon \Ho\Calg \to \Ho\Calg$ to $\Ho\Aalg$. 
\end{cor}
\begin{proof}  The adjunction (\ref{eq adjunction of A in C}) descends to the homotopy categories and gives and adjunction:
 	\begin{equation}
	\MorHo{\Aalg}{D^{\mathrm{ab}}, B} \cong \MorHo{\Calg}{D, iB},   
	\end{equation}
for $D \inn \Calg,\, B \inn \Aalg$. See also \cite[Adjoint functor lemma]{MR0341469}. The rest of the statements are clear.
\end{proof}

Consequently, we see that $\SW\Aalg$ is a full triangulated subcategory of $\SW\Calg$.



\begin{lem}\label{lem loop on A} The loop-functor $\Omega\colon \Ho\Aalg \to \Ho\Aalg$ does not admit a left-adjoint.
\end{lem}
\begin{proof}
By Gelfand-Naimark duality, the category $\CM_{\pt}$ of pointed compact Hausdorff spaces is contravariantly equivalent to $\Aalg$, hence form a category of {\em cofibrant} objects. We need to show that the functor
	\begin{equation}\label{eq sigma}
	\Sigma = S^{1} \wedge - \colon \Ho{\CM_{\pt}} \to \Ho{\CM_{\pt}}
	\end{equation}
does not admit a {\em right}-adjoint. We show that, in fact, the functor 
	\begin{equation}\label{eq nonrep functor}
	\Ho{\CM_{\pt}} \to \cat{Set}_{\pt},\quad X \mapsto \MorHo{\CM_{\pt}}{\Sigma X, S^{1}}	
	\end{equation}
is {\em not} representable, where $\cat{Set}_{\pt}$ denote the category of pointed sets. Indeed, suppose that for some $Y \inn \Ho{\CM_{\pt}}$ we have a natural identification
	\begin{equation}
	\MorHo{\CM_{\pt}}{\Sigma X, S^{1}} \cong \MorHo{\CM_{\pt}}{X, Y}.
	\end{equation}
Let $\Top_{\pt}$ denote the category of pointed compactly generated weakly Hausdorff topological spaces. 
Then $\CM_{\pt}$ is a full (reflective) subcategory of $\Top_{\pt}$ and  $\Ho{\CM_{\pt}}$ is a full subcategory of $\Ho{\Top_{\pt}}$. Moreover, the functor $\Sigma$ of (\ref{eq sigma}) 
 is the restriction of 
 	\begin{equation}
	\Sigma = S^{1} \wedge - \colon \Ho{\Top_{\pt}} \to \Ho{\Top_{\pt}},
	\end{equation}
which does have a right-adjoint 
	\begin{equation}
	\Omega = \Mor{\Top_{*}}(S^{1}, - )\colon \Ho{\Top_{\pt}} \to \Ho{\Top_{\pt}}.  	
	\end{equation}
Hence for $X \inn \CM_{\pt}$, we have 
	\begin{align}
	\MorHo{\Top_{\pt}}{X, Y} &\cong \MorHo{\CM_{\pt}}{X, Y}\\
	 &\cong \MorHo{\CM_{\pt}}{\Sigma X, S^{1}}\\
	 &\cong \MorHo{\Top_{\pt}}{X, \Omega S^{1}}.
	\end{align}
Moreover, by Yoneda's Lemma, the natural identification above must be induced by a map $f\colon Y \to \Omega S^{1}$ of $\Top_{\pt}$. This is a contradiction, because, since $Y$ is compact $f$ cannot be surjective on $\pi_{0}$.
	
 
\end{proof}

\begin{cor}\label{cor no adjoint omega C} The loop-functor $\Omega\colon \Ho\Calg \to \Ho\Calg$ does not admit a left-adjoint. 
\end{cor}
\begin{proof} Suppose that $\Sigma\colon \Ho\Calg \to \Ho\Calg$ is a left-adjoint of $\Omega$. It follows that the composition 
	\begin{equation}
	(-)^{\mathrm{ab}}\circ \Sigma \circ i \colon \Ho\Aalg \to \Ho\Calg \to \Ho\Calg \to \Ho\Aalg	
	\end{equation}
is a left-adjoint of $\Omega \colon \Ho\Aalg \to \Ho\Aalg$, contradicting Lemma~\ref{lem loop on A}. 
\end{proof}

Now Theorem~\ref{thm no Quillen} follows from Lemma~\ref{lem omega adjoint} and Corollary~\ref{cor no adjoint omega C}.

\bibliographystyle{amsalpha}
\bibliography{../BibTeX/biblio}

\end{document}